\documentclass{amsart}
\usepackage{amssymb}
\usepackage{graphicx}
\usepackage[all]{xy}
\usepackage{a4wide}
\usepackage{xcolor}
\usepackage{hyperref}
\begin{document}

\renewcommand{\phi}{\varphi}
\newcommand{\gal}{\textnormal{Gal}}
\newcommand{\supp}{{\rm supp}}
\def\iff{\leftrightarrow}
\def\Iff{\Longleftrightarrow}
\newcommand{\Fr}{{\rm Fr}}
\newcommand{\Sym}{{\rm Sym}}
\newcommand{\Alt}{{\rm Alt}}
\def\ab{{\rm ab}}
\def\sol{{\rm sol}}
\def\and{{\rm and}}
\def\cd{{\rm cd}}
\def\alg{{\rm alg}}
\def\Aut{{\rm Aut}}
\def\Hom{{\rm Hom}}
\def\res{{\rm res}}
\def\id{{\rm id}}

\newcommand{\calC}{{\mathcal C}}
\newcommand{\calE}{{\mathcal E}}
\newcommand{\calO}{{\mathcal O}}

\newcommand{\fgag}{{\bar f}}
\newcommand{\Fgag}{{\bar F}}
\newcommand{\Rgag}{{\bar R}}
\def\sigmagag{{\bar \sigma}}

\def\Ehat{{\hat E}}
\def\Fhat{{\hat F}}
\def\Lhat{{\hat L}}
\def\Mhat{{\hat M}}
\def\phihat{{\hat \varphi}}

\newcommand{\Kgal}{{\tilde K}}

\newcommand{\bbA}{\mathbb{A}}
\newcommand{\bbC}{\mathbb{C}}
\newcommand{\bbF}{\mathbb{F}}
\newcommand{\bbN}{\mathbb{N}}
\newcommand{\bbP}{\mathbb{P}}
\newcommand{\bbQ}{\mathbb{Q}}
\newcommand{\bbZ}{\mathbb{Z}}
\newcommand{\bbZhat}{{\hat{\mathbb{Z}}}}

\newcommand{\pfrak}{{\mathfrak p}}
\newcommand{\Pfrak}{{\mathfrak P}}

\newcommand{\bfa}{{\mathbf a}}
\newcommand{\bfA}{{\mathbf A}}
\newcommand{\bfT}{{\mathbf T}}
\newcommand{\bfsigma}{{\boldsymbol{\sigma}}}
\newcommand{\bfomega}{{\boldsymbol{\omega}}}

\renewcommand{\theenumi}{\alph{enumi}}
\renewcommand{\labelenumi}{\textnormal{(\theenumi)}}

\newtheorem{lemma}{Lemma}[section]
\newtheorem{theorem}[lemma]{Theorem}
\newtheorem{corollary}[lemma]{Corollary}
\newtheorem{proposition}[lemma]{Proposition}
\newtheorem{conjecture}[lemma]{Conjecture}
\theoremstyle{remark}
\newtheorem{remark}[lemma]{Remark}
\newtheorem{question}[lemma]{Question}
\theoremstyle{definition}
\newtheorem{condition}[lemma]{Condition}
\newtheorem{necs-cond}[lemma]{Necessary Condition}
\newtheorem{definition}[lemma]{Definition}

\CompileMatrices 

\title{Irreducible values of polynomials}%

\author{Lior Bary-Soroker}%
\address{
Institut f\"ur Experimentelle Mathematik,
Universit\"at Duisburg-Essen,
Ellernstrasse 29,
D-45326 Essen,
Germany}
\email{barylior@math.huji.ac.il}%
%

\date{\today}%
\begin{abstract}
Schinzel's Hypothesis H is a general conjecture in number theory on prime values of polynomials that generalizes, e.g., the twin prime conjecture and Dirichlet's theorem on primes in arithmetic progression. We prove an arithmetic analog of this conjecture for polynomial rings over pseudo algebraically closed fields. This implies results over large finite fields. A main tool in the proof is an irreducibility theorems \`a la Hilbert.
\end{abstract}
\maketitle

\section{Introduction}
A necessary condition for a finite family of irreducible polynomials $f_1(X), \ldots, f_r(X) \in \mathbb{Z}[X]$ with positive leading coefficients to admit infinitely many simultaneous prime values in $\mathbb{Z}$ is that there are \emph{no local obstructions}, i.e., the function $x\mapsto f_1(x)\cdots f_r(x) \mod p$ is not the zero function, for all primes $p\in \mathbb{Z}$. Schinzel's Hypothesis H predicts that this condition is also necessary. Special cases of this conjecture are Dirichlet's theorem on primes in arithmetic progressions (taking $f_1 = aX + b$) and the twin prime conjecture (taking $f_1 = X-1$ and $f_2 = X+1$). Dirichlet's theorem is the only proven case of the conjecture, sieve methods are used to obtain near-misses of the conjecture (e.g., Brun's sieve and Chen's theorem).

The analogy between the rings $\mathbb{Z}$ and $\mathbb{F}_q[t]$, where $\mathbb{F}_q$ is the field of $q$ elements, suggests a na\"ive Hypothesis H analog for the ring $\mathbb{F}_{q}[t]$: For every finite family of irreducible polynomials $f_1(t,X), \ldots, f_r(t,X)\in \mathbb{F}_q[t,X]$ having no local obstructions there exist infinitely many simultaneous prime values in $\mathbb{F}_q[t]$. It is quite surprising that this conjecture fails, e.g., all the values of $X^8 + t^3$ in $\mathbb{F}_2[t]$ are composite \cite{Swan1962}. In \cite{ConradConradGross2008} Conrad, Conrad, and Gross study global obstructions coming from values of the M\"obius function. 

Nevertheless there are some general results for \emph{large} finite fields. In \cite{BenderWittenberg2005} Bender and Wittenberg find linear irreducible substitutions  under the assumption that the Zariski closure of each of the plane curves $f_i(T,X)$ in $\mathbb{P}^2$ is smooth, a technical assumption on the characteristic (in particular the characteristic is odd) and of course, that the cardinality of the field is sufficiently large comparable to the degrees of the polynomials.

In \cite{Pollack2008} Pollack proves an arithmetic type result. 

\begin{theorem}[Pollack]\label{thm:large-finite-arithmetic-type}
Let $n,B$ be positive integers, $p$ a prime such that $p\nmid 2n$, $q$ a power of $p$, $f_1(X), \ldots, f_r(X) \in \mathbb{F}_q[X]$ non-associate irreducible polynomials such that $\sum \deg(f_i)\leq B$. Then the number of degree $n$ monic $g(t) = t^n + \cdots \in \mathbb{F}_q[t]$ for which all of the $f_i(g(t))$ are irreducible in $\mathbb{F}_q[t]$ is 
\[
\frac{q^n}{n^r} + O_{n,B}(q^{n-\frac{1}{2}}).
\]
\end{theorem} 

The asserted constant in Pollack's theorem is of order of magnitude $(n!)^B$, hence the asymptotic is useless when $n\to \infty$, but it works when $n$ is fixed and $q\to \infty$. A conjectural formula for $q^n\to \infty$ is given in \cite{Pollack2008-conjecture}, there the main term depends on local data.

In this work we prove an arithmetic type theorem for \emph{PAC} fields satisfying an obvious necessary condition. A field $K$ is called \textbf{PAC} (short for \textbf{pseudo algebraically closed} field) if $V(K)\neq \emptyset$, for every absolutely irreducible non-void $K$-variety $V$. For example, Pop shows that the field $\mathbb{Q}_{{\rm tr}}(\sqrt{-1})$ we get by taking all totally real numbers and $\sqrt{-1}$ is PAC. Let $\sigma_1, \ldots, \sigma_e\in \gal(\mathbb{Q}) $ and let $\tilde{\mathbb{Q}} (\sigma_1, \ldots, \sigma_e)$ be the field fixed by all $\sigma_1, \ldots, \sigma_e$. Then Jarden proves that the probability that $\tilde{\mathbb{Q}} (\sigma_1, \ldots, \sigma_e)$ is PAC is $1$, w.r.t.\ the probability Haar measure of the profinite, hence compact, group $\gal(\mathbb{Q})$.  (Here $\tilde{\mathbb{Q}}$ is the field of all algebraic numbers, and $\gal(\mathbb{Q}) = \gal(\tilde{\mathbb{Q}}/\mathbb{Q})$ is the absolute Galois group of $\mathbb{Q}$.) 
\begin{theorem}
\label{thm:Ans_PAC}
Let $K$ be a PAC field of characteristic $p\geq 0$, let  $f_1, \ldots, f_r\in K[X]$ be non-associate irreducible separable polynomials with respective roots $\omega_1, \ldots, \omega_r$, and let $n$ be a positive integer, odd if $p=2$. Suppose that 
\begin{enumerate}
\renewcommand{\theenumi}{$\star$}
\item \label{cond:exists_sep_ext}
$K(\omega_i)$ has a separable extension of degree $n$, for $i=1,\ldots, r$. 
\end{enumerate}
Then there exists a Zariski dense set of $(a_1, \ldots, a_n)\in K^n$ such that, for $g(t) = t^n + a_1 t^{n-1} + \cdots + a_n$, all of the $f_{i}(g(t))$
are separable and irreducible in $K[t]$.
\end{theorem}

Note that \eqref{cond:exists_sep_ext} is necessary. Indeed,  assume $f_i(g(t))$ to be separable and irreducible. Choose a root $\eta_i$ of $g(t)-\omega_i$. Then $f(g(\eta_i)) = f_i(\omega_i) = 0$, hence $\eta_i$ is a root of $f_i(g(t))$. Moreover $K(\omega_i) \subseteq K(\eta_i)$ and 
\[
\deg(f_i\circ g)
=[K(\eta_i):K]
=[K(\eta_i):K(\omega_i)][K(\omega_i):K]
\le n\cdot\deg(f_i)
=\deg(f_i\circ g),
\]
so $[K(\eta_i):K]=n\cdot \deg f$ and $[K(\eta_i):K(\omega_i)]=\deg(g(t))=n$, as needed.

Theorem~\ref{thm:Ans_PAC} is more than an analog of Theorem~\ref{thm:large-finite-arithmetic-type}, it actually implies it, as explained below. Moreover, Theorem~\ref{thm:Ans_PAC} contains the `wild' case in odd characteristic, i.e., when $2\neq p \mid n$, and the `tame' case is characteristic $2$, i.e., when $p=2$ and $n$ is odd. Therefore it strengthens Theorem~\ref{thm:large-finite-arithmetic-type} to these important cases. 

A \textbf{pseudo finite} field $K$ is defined to be a perfect PAC field with $\gal(K) \cong \widehat{\mathbb{Z}}$. The latter condition implies in particular  \eqref{cond:exists_sep_ext} for any $\omega_i$. So we get the following 

\begin{corollary}\label{cor:pseudo_finite}
Let $n,B>0$ be fixed. Then any pseudo finite field $K$ of characteristic $p\geq 0$ satisfies the following elementary statement. For every irreducible polynomials $f_1(X), \ldots, f_r(X)\in K[X]$ satisfying $\sum \deg(f_i)\leq B$ there exists a Zariski dense set of $(a_1, \ldots, a_n) \in K^n$ such that, for $g(t) = t^n + a_1 t_{n-1} + \cdots + a_n$, all of the $f_i(g)$ are irreducible, \emph{provided} $n$  is odd if $p=2$. 
\end{corollary}

In \cite{Ax1968} Ax proves that an elementary statement is true for almost all finite fields if and only if it is true for all pseudo finite fields (see also \cite[\S20.10]{FriedJarden2008}). Therefore, an immediate consequence is that Corollary~\ref{cor:pseudo_finite} implies Theorem~\ref{thm:large-finite-arithmetic-type} in a weak sense, namely, it gives existence of $g$'s, but not the mentioned asymptotic. But in fact, while proving the theorem, a more technical statement is proved, from it Theorem~\ref{thm:large-finite-arithmetic-type} follows in its full strength using the Lang-Weil estimates, including the case $p=2$ and $n$ is odd, and the case $2\neq p\mid n$:

\begin{theorem}\label{thm:large-finite-arithmetic-type-2}
Let $n,B$ be positive integers, $p$ a prime, $q$ a power of $p$, $f_1(X), \ldots, f_r(X) \in \mathbb{F}_q[X]$ non-associate irreducible polynomials such that $\sum \deg(f_i)\leq B$. Assume $n$ is odd if $p=2$. Then the number of degree $n$ monic $g(t) \in \mathbb{F}_q[t]$ for which all of the $f_i(g(t))$ are irreducible in $\mathbb{F}_q[t]$ is 
\[
\frac{q^n}{n^r} + O_{n,B}(q^{n-\frac{1}{2}}).
\]
\end{theorem} 

An interesting special case is when $f_i = X + i -1$:
\begin{corollary}
Assume $q \gg r,n$ and that $n$ is odd if $q$ is even. Then there exists a monic polynomial $g(t) \in \mathbb{F}_q[t]$ of degree $n$ such that $g,g+1, \cdots, g+r-1$ are irreducible. 
\end{corollary}

When $r=1$, Pollack proves a result over small finite fields: Over a finite field $\mathbb{F}_q$ there are infinitely many $g$ such that $g, g+1$ are irreducible \cite[Theorem~4]{Pollack2008}. 

Another interesting case is $f_1 = X^2 +1$, which can be considered as an analog of Landau's problem. Recall that $X^2+1$ is irreducible in $\mathbb{F}_q$ if and only if $q\equiv 1 \mod 4$.

\begin{corollary}
Assume $q \gg n$ and that $n$ is odd if $q$ is even. Then there exists a monic polynomial $g(t)\in \mathbb{F}_q[t]$ of degree $n$ such that $g(t)^2+1$ is irreducible.
\end{corollary}

A natural way to try to prove Theorem~\ref{thm:Ans_PAC} is the following. (For simplicity assume $r=1$ and $f=f_1$.) Let $\mathcal{G}(\bfA,t) = t^n + A_1 t^{n-1} + A_{2} t^{n-2} + \cdots + A_n$ be a generic polynomial, i.e., $\bfA=(A_1, \ldots, A_n)$ is an $n$-tuple of  algebraically independent variables. Consider the polynomial $\mathcal{F}(\bfA,t) = f(\mathcal{G}(\bfA,t))$. It is easy to show that this polynomial is irreducible. So the proof reduces to the following question: Can we specialize $\bfA \mapsto \bfa \in K^n$ so that irreducibility is preserved? 

If $K$ is a number field (more generally, Hilbertian field), then Hilbert's irreducibility theorem gives the desired irreducible specialization. On the other contrast, if $K$ is algebraically closed, it is obvious that  no  irreducible specialization exists, if the $t$-degree of $\mathcal{F}$ is greater than $1$. 

We prove here a \emph{weak Hilbert's irreducibility} theorem for PAC fields, that gives a necessary and sufficient condition for a polynomial over a PAC field to have a Zariski dense set of irreducible specializations (in fact we prove a more general result, see Theorem~\ref{thm:HITPAC}).  The condition is given in terms of embedding problems. 

In order to apply the weak Hilbert's irreducibility theorem in our situation we have  to calculate the Galois group of $\mathcal{F}(\bfA, t)$. We show that for a positive integer $n$, odd if $p=2$, we have 
\[
\gal(\mathcal{F}, K(\bfA)) \cong S_n \wr_\Omega \gal(f,K).
\]
Here $\Omega$ is the set of roots of $f$, $\gal(f,K)$ acts on $\Omega$ in a natural way, and $S_n\wr_\Omega \gal(f,K)$ is the permutational wreath product. It is a little bit technical to show that this calculation is equivalent to the following  result, which may be of interest by itself:

\begin{proposition}
\label{prop:prod-sym-grps-oddchar}
Let $\tilde{K}$ be an algebraically closed field of characteristic $p\geq 0$, let $n$ be a positive integer, odd if $p=2$, let $\Omega\subseteq \tilde{K}$ be a finite set, let $\mathcal{G}(\bfA,t) = t^n + A_1 t^{n-1} + A_{2} t^{n-2} + \cdots + A_n$ be a generic polynomial. Then the splitting fields $F_\omega$ of $\mathcal{G}-\omega$, $\omega\in \Omega$ are linearly disjoint over $\tilde{K}(\bfA)$ and 
\[\gal\Big( \prod_{\omega\in \Omega}(\mathcal{G} -\omega)\Big) \cong S_n^\Omega.\]  
\end{proposition}

It surprised the author to find out that these results fails when $p=n=2$. In this case the assertion holds true if and only if $\sum_{i=1}^{2\ell} \omega_i \neq 0$ for every \emph{even} number of distinct elements $\omega_1,\ldots, \omega_{2\ell}\in \Omega$. If $n = 2k > 2$, we suspect this result to fail, but we do not know to prove it, or  to give an exact condition for it to hold, as we had for $n=2$. From this it follows that Theorem~\ref{thm:Ans_PAC}, and hence Theorem~\ref{thm:large-finite-arithmetic-type-2} holds true when $p=n=2$.

We conclude the introduction with a remark. It is interesting to consider Schinzel's Hypothesis H for polynomial rings over other fields. This can be done for fields having a \emph{PAC extension} that satisfies \eqref{cond:exists_sep_ext}. Examples of such interesting families of fields can be found in \cite{Bary-Soroker2009PAMS, Bary-SorokerKelmer}. This will be dealt somewhere else. 

\subsection*{Acknowledgments}
I thanks Moshe Jarden for pointing out several mistakes in an earlier version, for many remarks and suggestions that contributed to the presentation of paper, Peter M\"uller for letting me know about Pollack's work, and Wulf-Dieter Geyer, Dan Haran, Ehud Hrushovski, Christian Kappen, Zeev Rudnick, and Tomer Schlank for helpful discussions.

Part of this work was done while the author was a Lady Davis postdoc fellow in the Hebrew University of Jerusalem. The author is an Alexander von Humboldt postdoc fellow in The Instituts f\"ur Experimentelle Mathematik in Duisburg-Essen University. This research is partially supported by a grant from the ERC.

\section{Weak Hilbert's irreducibility theorem}
\subsection{Geometric embedding problems}
Let $K$ be a field, $V,W$ irreducible smooth affine $K$-varieties , and $\rho\colon W\to V$ a finite separable morphism. Let $R,S$ be the respective rings of regular functions of $V,W$, consider $R$ as a subring of $S$, and let $F/E$ be the corresponding function field extension. 

Assume that $V$ is absolutely irreducible and that $F/E$ is Galois. Then the field of constants of $V$ is $K$ and the field of constants $L=F\cap \tilde K$ of $W$  is Galois over $K$. (Here $\tilde K$ is a fixed algebraic closure of $K$.)  We have the restriction of automorphisms map $\alpha\colon \gal(F/E)\to \gal(L/K)$. This data defines an embedding problem for $K$ that we call a \textbf{geometric embedding problem}: 

\vskip5pt
\begin{minipage}{3.5cm}
$\qquad\mathcal{E}(W/V) :=$ 
\end{minipage}
\begin{minipage}{5cm}

$\xymatrix{
					&\gal(K) \ar[d]^{\pi}\ar@{.>}[dl]_{\theta=\mathfrak{P}^*}\\
\gal(F/E)\ar[r]^-{\alpha}		&\gal(L/K).
}
$
\end{minipage}
\vskip5pt

We note that a geometric embedding problem is a birational object, hence it depends only on the function field extension $F/E$. Therefore it coincides with the geometric embedding problem as defined in \cite{Bary-Soroker2009PACEXT} in terms of function fields. In this work it is more natural to work with varieties since we are interested in irreducible specializations of polynomials.

A \textbf{weak solution} of $\mathcal{E}(W/V)$ is a homomorphism $\theta \colon \gal(K)\to \gal(F/E)$ in the category of profinite groups, i.e.\ continuous. A weak solution is \textbf{proper} if it is  surjective. In this paper we will mainly have weak solutions, hence we decided to abbreviate the term `weak solution' and write in short `solution'. 

Rational points on $V$ induce  solutions in the following way. 
Let $\mathfrak{p}\in V(K)$ be a $K$-rational point of $V$ that is \'etale in $W$ and let $\mathfrak{P}\in \rho^{-1}(\mathfrak{p})$. By abuse of notation we also denote by $\mathfrak{p}$ the corresponding maximal ideal of the local ring $R_{\mathfrak{p}}$ at $\mathfrak{p}$, and we use similar notation for $\mathfrak{P}$. Then we have a homomorphism $\mathfrak{P}^*\colon \gal(K) \to \gal(F/E)$ defined by
\begin{equation}
\label{eq:geo-sol-act}
\Pfrak^*(\sigma)(x) \mod \mathfrak{P} = \sigma (x\mod \mathfrak{P}) , \qquad \mbox{for all } x\in S_{\mathfrak{P}}.
\end{equation}
The map $x\mapsto x\mod \mathfrak{P}$ bijectively maps $L\subseteq S$ to $L\subseteq S_{\mathfrak{P}}/\mathfrak{P}$, hence for $x\in L$ we have $\mathfrak{P}^*(\sigma)(x)  =  \sigma(x)$, so $\pi = \alpha\circ\mathfrak{P}^*$. In other words, $\mathfrak{P}^*$ is a solution of the embedding problem $\mathcal{E}(W/V)$. This solution is continuous since it factors through $\gal(K(\mathfrak{P})/K)$, where $K(\mathfrak{P}) = S_{\mathfrak{P}}/\mathfrak{P}$ is the residue field at $\mathfrak{P}$. A solution that equals $\mathfrak{P}^*$ for some $\mathfrak{P}$ as above is said to be \textbf{geometric}. The question whether a given solution is geometric is difficult, and can be considered as a finite version of Grothendieck's section conjecture. 

We note that if $\Phi$ is a place of $F$ that extends the map $S\to S/\mathfrak{P}$ and such that the residue field of $E$ is $K$, then the geometric solution $\Phi^*$ that is defined in \cite{Bary-Soroker2009PACEXT} coincides with $\mathfrak{P}^*$, since both are defined by the same formula. Therefore, here we present a different formulation of the same notion.  

If $\mathfrak{Q}\in \rho^{-1}(\mathfrak{p})$, then there exists $\tau \in \gal(F/EL) = \ker\alpha$ such that $\mathfrak{Q} = \tau\mathfrak{P}$. We thus get by  \eqref{eq:geo-sol-act}  that $\mathfrak{Q}^* = \tau \mathfrak{P}^* \tau^{-1}$. Vice-versa, every $\tau \mathfrak{P}^* \tau^{-1}$ comes from $\mathfrak{Q}\in \rho^{-1}(\mathfrak{p})$. Therefore it makes sense to define $\mathfrak{p}^*$ to be the $\ker\alpha$-inner-automorphism class  $\{\mathfrak{P}^*\mid \mathfrak{P}\in \rho^{-1}(\mathfrak{p})\}$. We call this class the \textbf{Artin class of geometric solutions}, this name is derived from the special case where $K$ is finite:

Assume that $K$ is a finite field consisting of $q$ elements. Then the absolute Galois group of $K$ is generated by a distinguish element, namely the Frobenius automorphism ${\rm Frob}\colon x\mapsto x^q$. 
For each \'etale $\mathfrak{P}\in \rho^{-1}(\mathfrak{p})$ we set $[W/V,\mathfrak{P}/\mathfrak{p}] = \mathfrak{P}^*({\rm Frob})$ and call $[W/V,\mathfrak{P}/\mathfrak{p}]$ the Frobenius element at $\mathfrak{P}$. Then we define the \textbf{Artin symbol} $(W/V,\mathfrak{p})$ as the conjugacy class of all Frobenius elements of $\mathfrak{P} \in \rho^{-1}(\mathfrak{p})$. 
Then the map $\mathfrak{p}^* \mapsto (W/V,\mathfrak{p})$ is a bijection, when the base field is finite. 

Next we describe special kind of geometric embedding problems that are associated to polynomials.
Let $V$ be be as above and let $f\in R[X]$ be a separable monic polynomial. We let $F$ be the splitting field of $f$  in a fixed algebraic closure of $E$, $\Omega$ the set of all the roots of $f$. Then we let $V_f = {\rm Spec}(S)$, where $S$ is a the integral closure of $R$ in $F$. Let $\rho\colon V_f \to V$ be the corresponding map. Then $\rho$ is a finite separable morphism and $F/E$ is Galois. We call $\mathcal{E}(f,V) := \mathcal{E}(V_f/V)$ the geometric embedding problem \textbf{associated to  $f$}. Note that the Galois group $\gal(F/K(V))$ naturally acts on $\Omega$, so it is a degree $n:=\deg f$ group. 

A rational point, $\mathfrak{p}\in V(K)$, is \'etale in $V_f$ if and only if  the discriminant of $f$ is invertible at $\mathfrak{p}$ \cite[Corollary 3.16]{Milne1980}.  Hence if and only if $f\mod \mathfrak{p}$ is a separable polynomial. 

\subsection{Factorizations of polynomials under a specialization map}

The \textbf{orbit type} of a solution $\theta \colon \gal(K) \to \gal(F/E)$ of $\mathcal{E}(f,V)$ is the partition of $n$ defined by the lengths of the $\theta(\gal(K))$-orbits of the roots of $f$. If there is only one orbit, we say that the solution is \textbf{transitive}. Since the orbit type is invariant under inner-automorphisms, it makes sense to define the orbit type of a class of $\ker\alpha$-inner-automorphisms of solutions to be the orbit type of one solution in the class.

The \textbf{factorization type} of a separable polynomial of degree $n$ is  the partition of $n$ defined by the degrees of  its irreducible factors. A basic fact in Galois theory is that the factorization type of a separable polynomial equals the orbit type of its Galois group.

The next lemma connects the factorization type of a specialized polynomial with the orbit type of the image of a geometric solution of the associated embedding problem. 

\begin{lemma}\label{lem:factorizatiotype}
Let $V = {\rm Spec} (R)$ be an absolutely irreducible smooth affine $K$-variety, let $f(X) \in R[X]$ be a monic separable polynomial, and let $\pfrak\in V(K)$ be \'etale in $V_f$. Then the factorization type of $f \mod \pfrak$ equals the orbit type of $\pfrak^*$. 

In particular, $f\mod \pfrak$ is irreducible if and only if $\pfrak^*$ is transitive.
\end{lemma}

\begin{proof}
Choose some $\Pfrak\in \rho^{-1} (\pfrak)$. Since $\mathfrak{p}$ is \'etale in $V_f$, the discriminant of $f$ is invertible at $\pfrak$, so the map $S_{\mathfrak{P}}\to S_{\mathfrak{P}}/\mathfrak{P}$ induces a bijection between the roots of $f$ and the roots of $f\mod \pfrak$. By \eqref{eq:geo-sol-act}, the action of $\gal(K)$ on the roots of $f\mod \pfrak$ coincides with the action of $\Pfrak^*(\gal(K))$ on the roots of $f$. Hence the assertion. 
\end{proof}

In the following result given a geometric embedding problem $\mathcal{E}(W/V)$ and a solution $\theta$ we give a form $\widehat{W}$ of $W$ with a correspondence between rational points on $\widehat{W}$ and points on $W$ inducing $\theta$. 

\begin{proposition}
\label{prop:geo-sol-rat-place}
Consider a geometric embedding problem $\mathcal{E}(W/V)$ as defined at the beginning of this section. 
Let $\theta \colon \gal(K) \to \gal(F/E)$ be a solution of $\mathcal{E}(W/V)$ and let $M = \tilde{K}^{\ker\theta}$ be the solution field. Then 
\begin{enumerate}
\item \label{part_1}
$W\times_K M$ factors to a disjoint union of absolutely irreducible components, let $W_M$ denote one of them.
\item \label{part_2}
 There exists an absolutely irreducible smooth variety $\widehat{W}$ with a diagram of finite separable morphisms
\[
\xymatrix{
W\ar[d]^{\rho}				&W\times_K M\ar[l]_-{{\rm p}}		&W_M \ar@{_(->}[l]_-{\iota}\ar[d]^\nu\\
V							&{\widehat{W}}\ar[l]_{\pi}						&{\widehat{W}}\times_K M \ar[l]_{\hat{\rm p}}
}
\] 	
such that $\deg \pi = |\ker\alpha|$, $\nu$ is an isomorphism, and for any $\mathfrak{Q}\in W_M(\tilde K)$ the following holds: 

\begin{quote}
$\mathfrak{P}^* = \theta$, for $\mathfrak{P}={\rm p} (\iota(\mathfrak{Q}))$ if and only if  $\hat{\mathfrak{P}} \in \widehat{W}(K)$, for $\hat{\mathfrak{P}} = \hat{{\rm p}}(\nu(\mathfrak{Q}))$,
\end{quote} 
provided $\mathfrak{P}$ is \'etale over $V$ (equivalently, $\mathfrak{Q}$ \'etale over $V$). 

\item \label{part_3}
Let $\Theta$ be the $\ker\alpha$-inner-autuomorphism class of $\theta$, and let $U\subseteq V(K)$ be the set of all $\mathfrak{p}\in V(K)$ that are \'etale in $W$ and $\mathfrak{p}^*= \Theta$. Then $\pi(\widehat{W}(K)) = U$ and, for every $\mathfrak{p}\in U$, $|\pi^{-1}(\mathfrak{p}) \cap \widehat{W}(K)| = \frac{|\ker \alpha|}{|\Theta|}$. 
\end{enumerate}
\end{proposition}

\begin{proof}
The kernel $\gal(M)$ of $\theta$ contains the kernel $\gal(L)$ of $\alpha$, so $L\subseteq M$. Therefore $W\times_K M$ factors to absolutely irreducible components. These components are disjoint because $W$ is smooth and isomorphic because the function field extension $F/E$ of $W\to V$ is Galois. Since $W\times_K M = {\rm Spec} (S\otimes_K M)$ and $SM$ is regular over $M$, the canonical map $S\otimes_K M \to SM$ defined by $s\otimes m \mapsto sm$ defines an embedding of the absolutely irreducible smooth variety $W_M = {\rm Spec}(SM)$ into $W\otimes_K M$. So $\iota(W_M)$ is an absolutely irreducible factor of $W\otimes_K M$.

We prove \eqref{part_2} using \cite[Proposition~3.2]{Bary-Soroker2009PACEXT}, where the assertion is proved in the language of function fields: Let $\hat F = FM$. In the proof of  \cite[Proposition~3.2]{Bary-Soroker2009PACEXT} a separable extension $\hat E\subseteq \hat F$ of $E$ is constructed with the following properties: 
\begin{enumerate}\renewcommand{\theenumi}{\arabic{enumi}}
\item \label{en1}
$\hat{E}$ is regular over $K$.
\item \label{en3}
$\gal(\hat{F}/E) \cong G\times_{\bar{G}} H$, where $G=\gal(F/E)$, $\bar{G} = \gal(L/K)$, and $H = \gal(EM/E)\cong  \gal(M/K) \cong \theta(\gal(K))$. So $\theta$ induces an embedding $\bar\theta\colon H\to G$. 
\item \label{en4}
$\gal(\hat{F}/\hat{E}) \cong \Delta = \{(\bar\theta(h),h)\mid h\in H\}$.
\item \label{en5}
A $K$-rational place $\phi$ of $E$ extends to a place of $\Phi$ of $F$  with $\Phi^* = \theta$ if and only if $\phi$ extends to a $K$-rational place $\hat{\Phi}$ of $\Ehat$. In fact, something slightly stronger appears in the proof:
\item\label{en6}
 For a place $\Psi$ of $FM$ that is trivial on $M$ we have $\Phi^* = \theta$, for $\Phi=\Psi|_F$, if and only if, $\hat\Phi = \Psi|_{\hat{E}}$ is $K$-rational. 
\end{enumerate}

From \eqref{en3} and \eqref{en4} we have
\begin{enumerate}
\renewcommand{\theenumi}{\arabic{enumi}}
\addtocounter{enumi}{5}
\item \label{en2'}
$\displaystyle [\hat E : E] =  \frac{|G\times_{\bar G} H|}{|\Delta|} = \frac{|G\times_{\bar G} H|}{|H|}=|\ker \alpha|.$
\end{enumerate}
Since $\Delta \cap (\ker(G\times_{\bar{G}} H \to H)) =1$ and $\ker(G\times_{\bar{G}} H \to H) \cong \gal(\Fhat/EM)$, Galois correspondence implies that 
\begin{enumerate}\renewcommand{\theenumi}{\arabic{enumi}}\addtocounter{enumi}{6}
\item\label{en7}
 $\hat{E} (EM) = \hat{E} M =\hat{F}$.
\end{enumerate}

Recall that $R$ (resp.\ $S$) is the ring of regular functions on $V$ (resp.\ $W$), and since $V,W$ are smooth, $S$ is the integral closure of $R$ in $F$. Let $\hat{S}$ be the integral closure of $R$ in $\hat{E}$. By \eqref{en1} we have that $\hat{S} \otimes_K M \cong \hat{S} M$, and by \eqref{en7} we get that $\hat{S} M = SM$. So we have the following  
ring extension diagram.
\[
\xymatrix{
S\ar[r]^-{{\rm p}'}				&S\otimes_{K} M\ar@{->>}[r]^{\iota'}	 		&SM\\
R\ar[r]^{\pi'}\ar[u]				&\hat{S}\ar[r]^{\hat{{\rm p}}'}		&\hat{S} M\ar@{=}[u] 
}
\]
Taking spectra we get the diagram of varieties given in \eqref{part_2} (where $\widehat{W} = {\rm Spec} (\hat{S})$). Then $\widehat{W}$ is absolutely irreducible by \eqref{en1} and $\deg \pi = |\ker\alpha|$ by \eqref{en2'}.

Let $\mathfrak{Q} \in W_M(\tilde K)$ be \'etale over $V$ and denote by $\mathfrak{P}$ and $\hat{\mathfrak{P}}$ the respective images  of $\mathfrak{Q}$ in $W$ and $\widehat{W}$. 
First assume that $\hat{\mathfrak{P}} \in \widehat{W}(K)$. Then $\hat{S}/\hat{\mathfrak{P}} = K$ (by abuse of notation $\hat{\mathfrak{P}}$ also denotes the ideal of $\hat{S}$ associated to the point). Extend the map $\hat{S}\to\hat{S}/\hat{\mathfrak{P}}$ to a $K$-rational place $\hat{\Phi}$ of $\hat{S}$. Extend $\hat{\Phi}$ linearly to a place $\Psi$ of $\hat{F}$. Then $\Phi = \Psi|_{F}$ extends the map $S\to S/\mathfrak{P}$. As explained in the definition of geometric solutions, $\Phi^*$ and $\mathfrak{P}^*$ coincides, so $\mathfrak{P}^* = \Phi^* = \theta$ by \eqref{en6}. 

Conversely, assume that $\mathfrak{P}^* = \theta$. Extend $SM \to SM/\mathfrak{Q}$ to a place $\Psi$ of $\hat{F}$, and let $\hat{\Phi}$ and $\Phi$ the restrictions to $\hat{E}$ and $F$, respectively. This can be done such that the residue field of $\Phi$ is $S/\mathfrak{P}$. Then $\Phi^* = \mathfrak{P}^* = \theta$, so by \eqref{en6}, $\hat{\Phi}$ is $K$-rational, and in particular, $\hat{S}/\hat{\mathfrak{P}}  = K$, so $\hat{\mathfrak{P}}\in \widehat{W}(K)$. This finishes the proof of \eqref{part_2}.

Finally we prove \eqref{part_3}. Let $\mathfrak{p}\in V(K)$ be \'etale in $W$, and assume $\mathfrak{p} = \pi(\hat{\mathfrak{P}})$, for $\hat{\mathfrak{P}}\in \widehat{W}(K)$. Choose $\mathfrak{Q}\in W_M$ lying above $\hat{\mathfrak{P}}$ and let $\mathfrak{P} = {\rm p}(\iota(\mathfrak{Q}))$. Then $\mathfrak{P}^* =\theta$ by \eqref{part_2}, so $\mathfrak{p}^* = \Theta$, and thus $\mathfrak{p}\in U$. Thus $\pi(\widehat{W}(K)) \subseteq U$.

Note that since $W_M$ is an irreducible factor of $W\times_K M$, every point of $W$ can be uniquely lifted to $W_M$. Similarly for $\widehat{W}$, since $W_M \cong \widehat{W} \times_K M$.  Let $\mathfrak{p}\in U$, so $\mathfrak{p}^* = \Theta$. Note that $\ker\alpha$ acts on $\rho^{-1}(\mathfrak{p})$, and $(\sigma\mathfrak{P})^* = \sigma \mathfrak{P}^* \sigma^{-1}$. So, since there are $|\Theta|$ such solutions, the set $P$ of $\mathfrak{P}$ with $\mathfrak{P}^* = \theta$ is of size $d=|\ker\alpha|/|\Theta|$. Then $P_M=({\rm p} \circ \iota)^{-1}(P)$ is also of size $d$, and hence $\hat{P} = \hat{\rm p}(\nu(P_M))$ is of size $d$. 

By \eqref{part_2}, 
$\hat{P}\subseteq \pi^{-1}(\mathfrak{p})\cap \widehat{W}(K)$. Assume $\hat{\mathfrak{P}}\not\in \hat{P}$. Take $\mathfrak{Q}\in W_M$ lying above it and take $\mathfrak{P}={\rm p}(\mathfrak{Q})$. Then $\mathfrak{P}^* \neq \theta$. So $\hat{\mathfrak{P}}\not\in \widehat{W}(K)$. 
\end{proof}

\subsection{Irreducibility theorem for PAC fields}
In this section we study conditions for a polynomial $f(X)\in R[X]$ that is irreducible and separable to admit irreducible  specializations. First we need the following consequence of Proposition~\ref{prop:geo-sol-rat-place}.
\begin{proposition}
\label{prop:geom-sol-PAC}
A field $K$ is PAC if and only if every solution $\theta$ of every geometric embedding problem $\mathcal{E}(W/V)$ is geometric. Moreover, for each $\theta$ there exists a Zariski dense set of $\mathfrak{p}\in V(K)$, \'etale in $W$, such that $\mathfrak{p}^* = \Theta$, where $\Theta$ is the $\ker\alpha$-inner-automorphism class of $\theta$. 
\end{proposition}

\begin{proof}
Let $\mathcal{E}(W/V)$ be a geometric embedding problem, and let $\theta$ be a solution. Every open subvariety of $\widehat{W}$ given in Proposition~\ref{prop:geo-sol-rat-place} has $K$-rational points, hence $\widehat{W}(K)$ is Zariski dense. The assertion follows, since for every $\mathfrak{p}\in \pi(\widehat{W}(K))\subseteq V(K)$ that is \'etale in $W$ we have $\mathfrak{p}^* = \Theta$. 

Vice-versa, let $V$ be an absolutely irreducible $K$-variety, we can assume that $V$ is smooth, otherwise we replace $V$ by an open subvariety. Consider the geometric embedding problem $\mathcal{E}(V/V)$. It has a solution, the trivial one, say $\theta\colon \gal(K)\to 1$. By assumption $\theta =\mathfrak{p}^*$, for some $\mathfrak{p}\in V(K)$. In particular, $V(K)$ is not empty, and $K$ is PAC. 
\end{proof}

Let $V$ be an absolutely irreducible smooth $K$-variety, $R$ the ring of regular functions, and $f(X)\in R[X]$ a monic separable polynomial. By Lemma~\ref{lem:factorizatiotype} to have $\mathfrak{p}\in V(K)$ such that $f \mod \mathfrak{p}$ is of a given factorization type, say $P$, it is necessary that $\mathcal{E}(f,V)$ has a solution with orbit type $P$. Over PAC fields this condition also suffices. 

 \begin{theorem}
\label{thm:HITPAC}
Let $K$ be a PAC field, $V$  an absolutely irreducible smooth $K$-variety with ring of regular function $R$, $f(X)\in R[X]$ a separable monic polynomial, and $P$ a partition of $\deg f$. Assume that the induced embedding problem has a solution whose orbit type is $P$. Then there exists a Zariski dense set of $\pfrak\in V(K)$ such that $f\mod \pfrak$ is a separable polynomial of factorization type $P$.
\end{theorem}

\begin{proof}
Let $\theta$ be a solution of factorization type $P$. By Proposition~ \ref{prop:geom-sol-PAC}  we have $\Theta = \mathfrak{p}^*$, for a Zariski dense set of $\mathfrak{p}\in V(K)$ that are \'etale in $V_f$. For each such $\mathfrak{p}$, the action of $\gal(K)$ on the roots of $f\mod \mathfrak{p}$ coincides (up to labeling of the roots) with the action on the image of $\theta$ on the roots of $f$. The latter has orbit type $P$, so the factorization type of $f\mod \mathfrak{p}$ is $P$.  
\end{proof}

\begin{remark}
From the proof follows a stronger statement. Namely, under the notation of the theorem, there exists a Zariski dense set of $\mathfrak{p}\in V(K)$ such that the splitting field of $f \mod \mathfrak{p}$ is the solution field $\tilde K^{\ker \theta}$ of $\theta$. 
\end{remark}

\begin{remark}
The theorem holds true even if $f$ is not monic, and $V$ is not smooth. Indeed, in that case we replace $V$ with a smooth open subvariety, such that, the leading coefficient of $f$ is invertible in the ring of regular functions. 
\end{remark}

To connect Theorem~\ref{thm:HITPAC} to Hilbertian fields, one needs to take $V=\mathbb{A}^n$, and take $P$ the partition to a single part. 

\begin{corollary}
Let $K$ be a PAC field, $f(A_1,\ldots, A_n,X)\in K[A_1,\ldots, A_n,X]$ a monic polynomial in $n+1$ variables, $n\geq 1$, that is separable in $X$, and let $g(A_1,\ldots, A_n)$ nonzero polynomial. Assume that the associated embedding problem $\mathcal{E}(f,\mathbb{A}^n)$ has a transitive solution. Then there exists $a_1,\ldots, a_n\in K^n$ such that 
\[
f(a_1, \ldots, a_n, X) \mbox{ is irreducible and } g(a_1,\ldots, a_n)\neq 0.
\]
\end{corollary}

An interesting special case, that appears in \cite{Bary-Soroker2009PAMS}, is when   $f(A_1,\ldots, A_n,X)$ is the `most' irreducible, i.e., when  $\gal(f,\tilde K(A_1,\ldots, A_n))$ is the full symmetric group. Then $\mathcal{E}(f,\mathbb{A}^n)$ is 
\[
\xymatrix{
		&\gal(K)\ar[d]\\
S_n\ar[r]&1,
}
\]
so a transitive solution exists if and only if $K$ has a separable extension of degree $n$.

\section{Calculating a Galois group of compositum of polynomials}
\subsection{Galois groups of shifted generic polynomials} 
Let us start by fixing the notation that will be used through out this section:
\begin{quote}\begin{tabular}{ll}
$K$&  an algebraically closed field of characteristic $p\geq 0$,\\[2pt]
$S_n$& the symmetric group on $n$ letters, \\[2pt]
$A_1, \ldots, A_n$& $n$ algebraically independent variables,\\[2pt]
$\Omega=\{\omega_1, \ldots, \omega_m \}\subseteq K$ & subset of $K$ of $m$ elements,\\[2pt]
$g(\bfA,X) = X^n + A_1 X^{n-1} + \cdots A_n$&  a general polynomial,\\[2pt]
$g_i(\bfA,X) = g (\bfA,X) -\omega_i$ &  shifted general polynomials, $i=1, \ldots, m$,\\[2pt]
$F_i$ & the splitting field of $g_i$ over $K(\bfA)$, $i=1,\ldots, m$.
\end{tabular}
\end{quote}

It is a simple exercise in algebra that $\gal(g,K(\bfA))  \cong S_n$ (e.g., \cite[Example 4, VI, \S2]{Lang2002}). Hence, $\gal(g_i(\bfA,X), K(\bfA)) = \gal(F_i/K(\bfA))  \cong S_n$. The objective of the section is to prove Proposition~\ref{prop:prod-sym-grps-oddchar}, i.e., that $F_1, \ldots, F_r$ are linearly disjoint.  Equivalently,  
\[
\gal\Big(\prod_{i=1}^m g_i , K(\bfA)\Big) \cong S_n^m
\] 
(since $\gal(F_i/K(\bfA)) = S_n$.)

This will be proved in a series of lemmas.

Assume for a short while that $p=n=2$. As mentioned in the introduction, the above assertion fails to hold. Nevertheless, one can characterize $\Omega$'s for which the assertion holds: $F_1, \ldots, F_m$ are linearly disjoint if and only if the sum of any \emph{even} number of $\omega_i$'s does not vanish. We leave this as an exercise for the reader. We do not know what happen when $p=2$ and $n>2$ is even. A na\"ive possibility is suggested in the following question.

\begin{question}
Assume $p=2$, $n$ even. Does the following assertion hold true? $F_1, \ldots, F_m$ are linearly disjoint if and only if every \emph{even} sum of distinct elements of $\Omega$ does not vanish. 

In the answer is no, characterize the finite subsets $\Omega\subseteq K$ for which $F_1, \ldots, F_m$ are linearly disjoint. 
\end{question}

We note that the assertion of Theorems~\ref{thm:large-finite-arithmetic-type} and \ref{thm:Ans_PAC} holds for polynomials $f_1, \ldots, f_r$ such that the set $\Omega$ of all their roots satisfies the assertion of the proposition, as will be  shown in the sequel. 

Now we start proving Proposition~\ref{prop:prod-sym-grps-oddchar}. We denote by $A_n$ the alternating group.

\begin{lemma}\label{lem:alt-sym}
Let $r,n\geq 1$, for each $i=1, \ldots,r$ let $\alpha_i \colon G\to S_n$ be an epimorphism, and let $\alpha = \prod \alpha_i \colon G \to n^r$. Assume $\beta \colon G\to (\mathbb{Z}/2\mathbb{Z})^r$ induced by the natural map $S_n^r\to S_n^r/A_n^r$ is surjective. Then $\alpha$ is surjective. 
\end{lemma}

\begin{proof}
The case $r=1$ is trivial; we proceed by induction on $r$. Let $S = S_n$ and $S' = S_n^{r-1}$. By induction we have that $\alpha' = \prod_{i=1}^{r-1} \alpha_i \colon G\to S'$ is surjective. Let $M'=\ker\alpha'$, $M=\ker\alpha_r$, and $A = G/M'M$. Let $S'  \to A$ and $S\to A$ be the projections induced by $\alpha', \alpha$, respectively. Then the image of $\alpha = \alpha'\times \alpha_r$ in $S'\times S$ is $S'\times_A S$. 

Since $G/M = S\cong S_n$, there exists $M\leq E \leq G$ such that $E/M \cong A_n$. Similarly, there exists $M'\leq E'\leq G$ such that $E'/M'\cong A_n^{r-1}$. 

$$
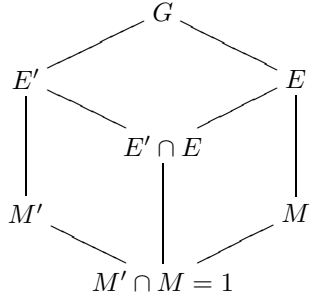
\xymatrix@=12pt{
						&G\ar@{-}[dr]\ar@{-}[dl]\\
			E'\ar@{-}[dd]\ar@{-}[dr]					&&E\ar@{-}[dl]\ar@{-}[dd]\\
						&E'\cap E\ar@{-}[dd]\\
			M'\ar@{-}[dr]							&&M\ar@{-}[dl]\\
				&M'\cap M =1			
}
$$

We have 
\begin{enumerate}
\item \label{S_n1}
$E/M$ contains all proper normal subgroups of $G/M$,
\item \label{S_n2} 
$E'/M$ is contained in every subgroup of $G/M'$ of index $2$, and 
\item  \label{S_n3}
$E'E=G$.
\end{enumerate}
Indeed, the proper normal subgroups of $S_n$ are $A_n$, $1$, and if $n=4$, the the Klein group $V_4$. Thus $A_n$ contains all proper normal subgroups, and    \eqref{S_n1} follows. 

For a group $H$ we let $H^{(2)}$ be the subgroup generated by squares. Then  every index $2$ subgroup contains $H^{(2)}$. Since $(h_1,h_2)^2 = (h_1^2,h_2^2)$ we have $(H_1\times H_2)^{(2)} = H_1^{(2)}\times H_2^{(2)}$. Note that $S_n^{(2)} = A_n$, since $S_n^{(2)}\subseteq A_n$, $(i\ j\ k)^2 = (i\ k\ j)$, and $A_n$ is generated by $3$-cycles. Thus $(S_n^{r-1})^{(2)} = A_n^{r-1}$, so $(G/M')^{(2)} = E'/M'$, and \eqref{S_n2} follows. 

We have $\ker(\beta) = E'\cap E$, so by assumption $G/E'\cap E \cong (\mathbb{Z}/2\mathbb{Z})^r$. Since $G/E' \cong( \mathbb{Z}/2\mathbb{Z})^{r-1}$ we have $E'E/E \cong E'/E'\cap E = \mathbb{Z}/2\mathbb{Z}$, so $(G:E'E) = (G:E)/(E'E:E) = 2/2 =1$, and \eqref{S_n3} follows. 

We continue with the proof. If $A=1$, then $G=S'\times S \cong S_n^r$, and we are done. Assume $A\neq 1$. Then $M'M/M$ is a proper normal subgroup of $G/M$, hence $M'M\leq E$ by \eqref{S_n1}.

In particular $M'\subseteq E$. Then $G/E \cong \mathbb{Z}/2\mathbb{Z}$ is a quotient of $G/M'$, so \eqref{S_n2}  implies $E'\leq E$. But then $E'E = E \neq G$, which contradicts \eqref{S_n3}. 
\end{proof}
 
For each $i$, let $E_i$ be the fixed field of $A_n$ in $F_i$. Then $E_i/K(\bfA)$ is a quadratic extension. Clearly if $F_1, \ldots, F_m$ are linearly disjoint, then  $E_1,\ldots, E_m$ are linearly disjoint. The next lemma shows that the converse holds. 

\begin{lemma}\label{lem:red-to-quad-ext}
If $E_1, \ldots, E_m$ are linearly disjoint over $K(\bfA)$, then $F_1, \ldots, F_m$ are linearly disjoint over $K(\bfA)$.
\end{lemma}

\begin{proof}
Let $\alpha_i \colon \gal(K(\bfA)) \to \gal(F_i/K(\bfA))$, let $\alpha = \prod_i \alpha_i\colon \gal(K) \to \prod_{i} \gal(F_i/K)\cong S_n^m$, and let $\beta \colon \gal(K) \to \prod_{i} \gal(E_i/K)\cong (\mathbb{Z}/2\mathbb{Z})^m$. By the assumption of the lemma, $\beta$ is surjective, so by Lemma~\ref{lem:alt-sym},  $\alpha$ is surjective. Thus $F_1, \ldots, F_r$ are linearly disjoint. 
\end{proof}

Let $h(X)$ be a separable polynomial of degree $n$ defined over a field $L$ and let $x_1, \ldots, x_n$ be its roots. We consider the Galois group of $h$ over $L$ as a permutation group on the roots of $h$. We let $E_h\subseteq L(x_1, \ldots, x_n)$ be the extension of degree at most $2$ which is the fixed field of all even elements in the Galois group of $h$.

If $p\neq 2$, then $E_h$ is generated by a root of $X^2 -\Delta(h)$, where $\Delta(h) = \prod_{i\neq j} (x_i - x_i)$ is the discriminant of the polynomial. Let $\delta(h)$ be the square class of $\Delta(h)$, i.e. $\delta(h) = [\Delta(h)]\in L^*/(L^*)^2 = H^1(L,\mathbb{Z}/2\mathbb{Z})$. The last equality follows from Kummer theory. This group is abelian of exponent $2$, hence we regard it as a vector space over $\bbF_2$. 

If $p=2$, then $E_h$ is generated by a root of the Artin-Schreier extension $X^2 + X + A(h)$, where $A(h) = \sum_{i\neq j} \frac{x_ix_j}{x_i^2 + x_j^2}$ (see, e.g., \cite{Berlekamp1976}). Then we let $\delta(h)$ be the Artin-Schreier coset of $A(h)$, i.e., $\delta(h) = A(h) + \wp(L)\in L /\wp(L) = H^1(L,\mathbb{Z}/2\mathbb{Z})$, where $\wp(x) = x^2 + x$. The last equality follows from Artin-Schreier theory. This group, as for $p\neq 2$, is abelian of exponent $2$, hence a vector space over $\bbF_2$. 

We call $\delta(h)$ the \textbf{discriminant class} of $h$. 

For each $i=1,\ldots, m$, $E_i = E_{g_i}$. Therefore $E_1, \ldots, E_m$  are linearly disjoint if and only if $\delta(g_i)$ are linearly independent. So Lemma~\ref{lem:red-to-quad-ext} can be reformulated in terms of the discriminant classes. 

\begin{lemma}\label{lem:red-to-disc}
If $\delta(g_1), \ldots, \delta(g_m)$ are linearly independent, then $F_1, \ldots, F_m$ are linearly disjoint over $K(\bfA)$.
\end{lemma}

General polynomials are complicated for calculations. Hence in the last lemma before the proof of the proposition, we shall reduced to specialized polynomials. 

\begin{lemma}
Let $\mathbf{a} = (a_1, \ldots, a_n)$ be an $n$-tuples in a field containing $K$, let $h (\bfa, X)= g(\bfa,X)\in K(\bfa)[X]$. For each $i=1,\ldots, m$ let $h_i = h-\omega_i$. Assume $h_1,\ldots, h_m$ are separable and $\delta(h_1), \ldots, \delta(h_m)$ are linearly disjoint. Then $\delta(g_1), \ldots, \delta(g_m)$ are linearly disjoint. 
\end{lemma}
We shall use this lemma only when $p\neq 2$, but for the sake of completeness, we proof it for arbitrary characteristic. 

\begin{proof}
We start with $p\neq 2$. 
Let $V = K(\bfA)^*/(K(\bfA)^*)^2$. For each $v\in V$, there exists a unique square-free polynomial $f_v(\bfA)$ such that $v = [f_v]$. For $u,v\in V$ we have $f_{u+v} = f_uf_v/d^2$, where $d=\gcd(f_u,f_v)$. 
Let 
\[
U = \{ v\in V \mid f_v(\mathbf{a}) \neq 0\}. 
\] 
Then $U$ is a subspace. Indeed, $f_{v+u}(\bfa) =  f_v(\bfa)f_u(\bfa)/d^2(\bfa)\neq 0$.

Let $W= K(\bfa)^*/(K(\bfa)^*)^2$. Let $T\colon U \to W$ be the map induced by $\mathbf{A} \mapsto \bfa$, i.e., $T(u) = [f_u(\bfa)]$. It is a linear map, because, by the above 
\[
T(v+u) = [f_v(\bfa) f_u(\bfa)/d^2(\bfa)] = [f_v(\bfa) f_u(\bfa)] = [f_v(\bfa)] + [f_u(\bfa)] = T(v) + T(u).
\]

Since there is a polynomial formula for the discriminant in term of the coefficients, it follows that $\Delta(h_i) = \Delta(g_i)(\bfa)$ (recall that $\Delta(g_i)\in K[\bfA]$). Let $f_i$ be the square-free part of $\Delta(g_i)$, so $\Delta(g_i)= f_i d^2$, and so $\Delta(h_i) = f_i(\bfa) d^2(\bfa)$. The assumption that $h_i$ is separable implies that $\Delta(h_i)\neq 0$. So we get $T(\delta(g_i))=[f_i(\bfa)] = \delta(h_i)$. This finishes the proof for $p\neq 2$.

Assume $p=2$. Let $R = K[\bfA, \Delta(g_1)^{-1}, \ldots, \Delta(g_m)^{-1}]$. Let $u=\frac{f(\bfA)}{k(\bfA)}\in K(\bfA)$ with $\gcd(f,k)=1$. Then the denominator  of $u^2 + u$ is $k^2(\bfA)$. In particular, if $u^2+u\in R$, then $u\in R$. Thus $R$ maps to $R/\wp(R)$ under the map $K(\bfA) \mapsto \wp(K(\bfA))$. In particular $\bfA\to \bfa$ induces a map $T\colon R/\wp(R)\to K(\bfa)/\wp(K(\bfa))$. 

As mentioned above, $\delta(g_i)$ is the class generated by $\sum_{l\neq l'} \frac{x_lx_{l'}}{x_l^2 + x_{l'}^2}$, where $x_1,\ldots, x_n$ are the roots of $g_i$. The common denominator of this expression is 
\[
\prod_{l\neq {l'}} (x_{l} +x_{l'}) =\prod_{l\neq {l'}} (x_{l} -x_{l'}) =\Delta(g_i).
\]
Therefore, $\delta(g_i)\in R$ and $\delta(h_i)\in T(R/\wp(R))$. This finishes the proof for $p=2$. 
\end{proof}

\begin{proof}[Proof of Proposition~\ref{prop:prod-sym-grps-oddchar}]
From the series of lemmas, it suffices to find $\bfA\mapsto \bfa$ such that $\delta(h_i)\in H^1(K(\bfa),\mathbb{Z}/2\mathbb{Z})$ are linearly independent, where $h_i  (X)= g_i(\bfa,X)$. It will be more convenient to specify the coefficients of $h(X) = X^n + a_1 X^{n-1} + \cdots + a_n$ instead of the sequence $(a_1, \ldots, a_n)$.  
We divide the proof into several cases. 

\noindent\textbf{The case when $p\nmid n$ and $n$ is even.} Consider $h(T,X)=X^n - T$, where $T$ is a variable over $K$. Then $\gal(h,K(T)) = C_n$ and the splitting field is $K(\sqrt[n]{T})$. Thus $K(\sqrt{T})$ is the unique quadratic subextension, and $\gal(K(\sqrt[n]{T})/K(\sqrt{T})) = C_{n/2} = C_n\cap A_n$. So  $\delta(h) = [T]$, hence $\delta(h_i) = [T - \alpha_i]$, which are obviously linearly independent in $K(T)^*/(K(T)^*)^2$. 

\noindent\textbf{The case when $p\nmid n(n-1)$.}
We shall use basic facts on ramification theory, one reference  that includes all we use is  \cite{Serre1979}.
Let $S,T$ be two algebraically independent variables, let $h_0(X) = X^n - S X^{n-1}$ and $h(T,X) = h_0(X) - T$. Let $K' = K(S)$,  $R=K'[T]$, let $U = K'(T)[X]/(h)\cong K'(x)$, where $h(T,x) = 0$. Since $T = h(x)$ we have $R[x] = K'[x]$, so $R[x]$ is the integral closure of $R$ in $U$. We then have that the different of the extension $U/K'(T)$ is the ideal generated by the derivative $h_0'(x) = nx^{n-2}(x-\frac{n-1}{n}S) $. So all the ramified primes are $(x)/(T)$, $(x-s_1)/(T-t_1)$, where $s_1 = \frac{n-1}{n}S$ and $t_1 = h_0(s_1)=S^n \beta$, $\beta\in K^*$. The ramification indices are  $n-1$ and $2$, respectively. Thus the corresponding inertia groups $I_0, I_1$ are generated by an $(n-1)$-cycle and by a transposition, respectively. 

This implies that (1) $I_1$ is not contained in $A_n$, and (2) $I_0$ is contained in $A_n$ if and only if $n$ is even. Thus the extension corresponding to the fixed field of the odd permutations is ramified at $(T-t_1)$ and if $n$ is odd also at $(T)$. Therefore, $\delta(h) = [(T-t_1)]$ if $n$ is even, and $=[(T)(T-t_1)]$ if $n$ is odd. In both cases, $\delta(h_i)$ are linearly independent in $K(S,T)^*/(K(S,T)^*)^2$. 

\noindent\textbf{The case when $n$ is odd and $p\mid n-1$.}
Similar argument as in the previous case shows that the discriminant classes $\delta(h_i)$ are linearly independent, where $h = X^n - S X^{n-2} - T$. 

\noindent\textbf{The case $p=2$ and $n$ is odd.} Here we use a different argument: 
Let $L = K(A_1, \ldots, A_{n-1})$, and let $U_i = L(A_n)[x_i]$, where $x_i\in F_i$ is a root of $g_i(\bfA,X)$, $i=1,\ldots, m$. As the ramification at infinity is tame, it suffices to show that the $A_n$-finite ramification loci of the $U_i/L(A_n)$ are distinct. (Indeed, then the finite ramification loci of the $F_i$ will be distinct, and hence, the $F_i$ will be linearly disjoint.)  

The different of $U_i/L(A_n)$ is contained in the ideal generated by $\frac{\partial g}{\partial X}(x_i)$. 
Thus the ramification points of  $U_i/L(A_n)$ are $A_n = g_i(u_1), \ldots, g_i(u_r)$, where $u_1, \ldots, u_r$ are the distinct roots of $\frac{\partial g_i}{\partial X}=\frac{\partial g}{\partial X}$ in a fixed algebraic closure of $L$. We have 
\[
\frac{\partial g}{\partial X} (X) = X^{n-1} + \sum_{j=1}^{\frac{n-1}{2}} A_{2i}X^{n-2j-1} = \left(X^{\frac{n-1}{2}} + \sum_{j=1}^{\frac{n-1}{2}} \sqrt{A_{2j}}X^{\frac{n-2j-1}{2}}\right)^{2} = \prod_{j=1}^r(X-u_j)^2,
\]
so $r=\frac{n-1}{2}$. Note that $u_1, \ldots, u_r$ are distinct, and furthermore, are algebraically independent variables because they are roots of a polynomial with variable coefficients. Since the coefficients of $\frac{\partial g}{\partial X}$ depend only on even indexed $A_i$, $u_1, \ldots, u_r$ are algebraically independent of $A_1, A_3, \ldots, A_{n}$. 

Substituting $u_j$ in the above equation, and multiplying by $u_j$ gives that
\[
{u_j}^n + A_2 u_j^{n-2} + \cdots + A_{n-1} u_j =0.
\]
Thus, for $j\neq j'$, 
\begin{eqnarray*}
g_i(\bfA,u_j) + g_{i'}(\bfA, u_{j'}) &=&  A_1 (u_j^{n-1} + u_{j'}^{n-1})  + A_3 (u_j^{n-3} + u_{j'}^{n-3})+ \cdots + (A_n + \alpha_{i}) + (A_n + \alpha_{i'}) \\
&=& A_1 (u_j^{n-1} + u_{j'}^{n-1})  + A_3 (u_j^{n-3} + u_{j'}^{n-3})+ \cdots + \alpha_{i}+ \alpha_{i'}
\not\in K.
\end{eqnarray*}
But $A_n = g_i(\bfA, u_1), \ldots, g_i(\bfA, u_r)$ are the ramification points of $U_i/L(A_n)$, so these are disjoint when $i$ varies. (Note that we used several times that $-1=1$ in $K$).

\noindent\textbf{The case when $p\neq 2$ and $p\mid n$.}
Take  $h(T,X) = X^n + \frac{1}{2}X^2 + T$.  We shall use that $\Delta(h)$ equals the resultant of $h$ and $h'$ and then we apply the formula for the resultant given by the determinant of the corresponding Sylvester matrix.  We note that $h' = X$. Then we have

 \[
\Delta(h) = 
\begin{vmatrix}
1		&\cdot		&\cdot		&\frac{1}{2}	&0		&T		&0		&\cdot		&0\\		&1			&			&		&\frac{1}{2}	&0		&T		&			&0\\
		&			&\ddots		&		&		&\ddots	& 		&\ddots		& \\
		&			&			&1		&\cdot	&\cdot	&\frac{1}{2}	&0			&T\\
0		&\cdot		&\cdot		&1		&0		&\cdot	&\cdot	&\cdot		&0\\
		&0			&\cdot		&\cdot	&1		&0		&\cdot	&\cdot		&0\\
		&			&0			&\cdot	&\cdot	&1		&0		&\cdot		&0\\
		&			&			&\ddots	&		&		&\ddots	&			&	\\
		&			&			&		&0		&		&		&1			&0
\end{vmatrix}
= \pm T.
\]
(We developed the determinant by the last column, then we got an upper triangular matrix with $1$'s on the diagonal.) 
Hence $\delta(h_i) = [T-\alpha_i]$, so the $\delta(h_i)$ are linearly disjoint.  
\end{proof}

\subsection{Galois group of composition of polynomials}
The symmetric group $\Sym(\Psi)$ is maximal in the family of all permutation groups on $\Psi$, in the sense the every permutation group on $\Psi$ embeds (by definition) into $\Sym(\Psi)$. We will show that the permutational wreath product plays a similar role, when considering permutation groups on $\Psi\times \Omega$ that respect the projection map $\Psi\times \Omega \to \Omega$. 

The Galois group of the composition of polynomials $\gal(f\circ g)$ maps onto the Galois group of the outer polynomial $\gal(f)$. We shall prove that if the inner polynomial is generic, then $\gal(f\circ g)$ is the maximal possible, namely the wreath product, namely the wreath product, provided $\deg g$ is odd of the characteristic is $2$. 

\begin{proposition}\label{prop:gene-comp-wreath}
Let $K$ be a field of characteristic $p\geq 0$, $n$ a positive integer, odd if $p=2$, $f(X)\in K[X]$ a separable polynomial, $g(\bfA,X)=X^n + A_1 X^{n-1} + \cdots + A_n$ a polynomial whose coefficients are variables, $\Omega, \Phi$ the sets of roots of $f, f\circ g$. Then, the splitting field of $f\circ g$ over $K(\bfA)$ is regular over the splitting field of $f$ and 
\[
\gal(f\circ g, K(\bfA)) \cong S_n \wr_\Omega \gal(f,K).
\]
This isomorphism respects the actions of the LHS on $\Phi$ and RHS on $\{1, \ldots, n\}\times \Omega$. 
\end{proposition}

Before proving this result we bring the formal definition of the permutational wreath product and two auxiliary results in Galois theory.

Let $H,G$ be finite groups acting on finite sets $\Psi, \Omega$, respectively.
Then $H^\Omega = \{ \zeta \colon \Omega \to H\}$ acts ``independently on each row'' of $\Phi := \Psi \times \Omega$, that is, $\zeta. (\psi,\omega) = (\zeta(\omega).\psi, \omega)$. We let $G$ act on the second coordinate of $\Phi$, namely, $g.(\psi,\omega)=(\psi, g.\omega)$, so $G$ permutes the ``columns". The group generated by these permutations is the \textbf{permutational wreath product}, the group structure is 
\[
H\wr_\Omega G = H^\Omega \rtimes G,
\]  
where $\zeta^{g^{-1}} (\omega) := (g.\zeta)(\omega) = \zeta (g.\omega)$, for $\zeta\in H^\Omega$, $\omega\in \Omega$, and $g\in G$. 
 
Each element of $H\wr_{\Omega} G$ has a unique representation as a product $\zeta g$, where $\zeta\in H^{\Omega}$ and $g\in G$. The multiplication is then given by $\zeta g \xi k = \zeta \xi^{g^{-1}} gk$, $\zeta,\xi\in H^{\Omega}$ and $g,k\in G$. Hence $\zeta^g = g^{-1} \zeta g$. The action of
$H\wr_{\Omega} G$ on $\Phi$ is given by:
\[
(\zeta g). (\psi,\omega) := (\zeta(g.\omega).\psi, g.\omega).
\]
The action is well defined because 
\[
(g\zeta^{g^{-1}}).(\psi,\omega) = g. (\zeta^{g^{-1}}(\omega).\psi, \omega) = (\zeta (g.\omega).\psi, g.\omega) = (\zeta g).(\omega,\psi).
\] 
The morphism $\zeta g \mapsto g\colon H\wr_\Omega G \to G$ respects the corresponding actions because the projection on the second coordinate of $\zeta g.(\psi,\omega)$ equals  $g.\omega$.

The following lemma is an exercise in basic Galois theory which we prove for the sake of completeness.  
\begin{lemma}\label{lem:gal-tower}
Let $r$ and $n$ be positive integers, $[n] =\{1,\ldots, n\}$, and $K\subseteq L_i\subseteq M_i$ a tower of finite separable extensions with $[M_i:L_i]\leq n$, for $i=1, \ldots, r$. Let 
\begin{itemize}
\item[] $N$ be the Galois closure of the compositum of $M_1, \ldots, M_r$ over $K$, 
\item[] $H= \gal(N/K)$,
\item[] $\Lhat_i$
the Galois closure of $L_i/K$,
\item[] $\Lhat$ the compositum of $\Lhat_1, \ldots, \Lhat_r$, 
\item[] $G = \gal(\Lhat/K)$, and
\item[] $\Omega = \coprod_{i=1}^r \Omega_i$, where $\Omega_i$ is the set of all embeddings  of $L_i$ to $N$ that fix $K$.
\end{itemize}
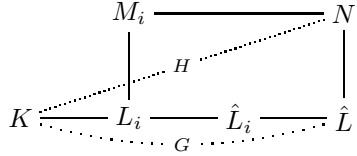
\[
\xymatrix{
					&M_i\ar@{-}	[rr]				&						&N\ar@{.}[dlll]|{~H~}
\\
K\ar@{-}[r]			&L_i\ar@{-}[r]\ar@{-}[u]		&\Lhat_i\ar@{-}[r]		&\Lhat\ar@{-}[u] \ar@{.}@/^10pt/[lll]|{~G~}
}
\] 
Then:
\begin{enumerate}
\item $G$ is a permutation group on $\Omega$ and there exists an embedding $\rho \colon H\to S_n \wr_{\Omega}G$ such that $\rho(\sigma) =\zeta_\sigma\sigma|_{\Lhat}$, for some $\zeta_\sigma\colon \Omega\to S_n$ and every $\sigma\in H$. 
\item
$[M_i:L_i] = n$ (for some $i\in [r]$) if and only if $H'= \rho(H)$ acts transitively on $[n]\times \Omega_i$.
\end{enumerate}
\end{lemma}

\begin{proof}
The Galois group $\gal(\Lhat_i/K)$ acts naturally on $\Omega_i$, namely $\sigma_i.\omega_i = \sigma_i\omega_i$. This action is faithful. 
By Galois correspondence, the assumption that $\Lhat$ is the compositum  of $\Lhat_1,\ldots, \Lhat_r$ is equivalent to  $\bigcap_{i=1}^r \gal(\Lhat/\Lhat_i)=1$.
The restriction maps $G\to \gal(\Lhat_i/K)$ induce an action of $G$ on each $\Omega_i$, and thus on $\Omega$.
An element $\sigma\in G$ fixes all elements $\omega\in \Omega_i$ if and only if $\sigma \in \gal(\Lhat/\Lhat_i)$. Therefore if $\sigma$ fixes all $\omega\in \Omega$, then $\sigma\in \bigcap_{i=1}^r \gal(\Lhat/\Lhat_i) = 1$. This implies that $G$ is  a permutation group on $\Omega$. 

For each $i$, let $\Phi_i$ be the set of all embeddings of $M_i$ to $N$ that  fix $K$. Then the same argument as above gives that $H$ is a permutation group on $\Phi=\coprod_{i=1}^r \Phi_i$. 

Let $\Psi_i\subseteq \Phi_i$ be the set of all embeddings of $M_i$ to $N$ that fix $L_i$. Then  $n\geq [M_i:L_i]= |\Psi_i|$. 
Let us enumerate the elements of $\Psi_i = \{\psi_{i1}, \ldots, \psi_{ik_i}\}$ ($k_i\leq n$). 
Let $T_i\colon \Omega_i\to H$ be a section of the restriction map, i.e., $T(\omega)|_{L_i} = \omega$, for all $\omega\in \Omega_i$, $i=1, \ldots, r$, and let $T = \coprod_i T_i\colon \Omega\to H$. (Note that, if $L_i = L_j$ for some $i\neq j$, we understand $\sigma|_{L_i}\in \Omega_i$ and $\sigma|_{L_j}\in \Omega_j$ as distinct elements of $\Omega$, although they induce the same function on $L_i$.) 

Then
{\renewcommand{\theenumi}{\roman{enumi}}
\begin{enumerate}
\item 
	for every $\psi_{ij}\in \Psi_i$, we have $\psi_{ij}|_{L_i} = \id_{L_i}$ (by definition), thus 
\item 
	$(\sigma T(\omega)\psi_{ij})|_{L_i} = \sigma|_{\Lhat} \omega$, for every $\sigma\in H$. 
\item
	If $\phi\in \Phi_i$, then $T({\phi}|_{L_i})^{-1} \phi \in \Psi_i$, hence it equals $\psi_{i J(\phi)}$ for some $J(\phi) \in [k_i]$.
\item 
	If $\omega\in \Omega_i$, $\psi_{ij}\in \Psi_i$, and $\sigma\in H$, 
	then $u:=\sigma T(\omega) \psi_{i j}\in \Phi_{i}$. Hence, by (ii) $u|_{L_i} = \sigma|_{\hat{L}} \omega$ and by (iii) $(T(u|_{L_i}))^{-1} u = \psi_{i J(u)}$. So $J(u)$ depends only on $j$, $\sigma$, and $\omega$ and for fixed $\sigma,\omega$ this dependency is injective. Therefore $J(u) = \zeta_\sigma(\omega).j$, for some $\zeta_{\sigma}(\omega)\in \Sym(k_i)\leq S_n$. 
\end{enumerate}
}

Consider the map $\rho^{*} \colon \Phi \to [n] \times \Omega$ defined by $\phi \in \Phi_i \mapsto (J(\phi),\phi|_{L_i})\in [n]\times \Omega_i$. By (iii), the map $(j, \omega)\mapsto T(\omega) \psi_{i j}$, for $\omega\in \Omega_i$ and $j\in [n]$ is the inverse of $\rho^*$, so $\rho^*$ is a bijection, and thus induces an embedding
\[
\rho \colon H\to \Sym([n]\times \Omega)
\]
that respects the actions of the former on $\Phi$ and the latter on $[n]\times \Omega$.  For (a) it suffices to show that the image $H'$ of $\rho$ is contained in the wreath product $S_n \wr_{\Omega} G$ (considered as a subgroup of $\Sym([n]\times \Omega)$).  And indeed, for $\sigma\in H$, if we write $\sigmagag = \sigma|_{\Lhat}$, then we have (using (i)-(iv))
\begin{eqnarray*}
\rho(\sigma).(j,\omega) 	&=& \rho^*(\sigma T(\omega)\psi_{ij}) = (J(\sigma T(\omega)\psi_{ij}),(\sigma T(\omega)\psi_{ij})|_{L_i}) \\
						&=& (\zeta_\sigma(\omega).j,\sigmagag\omega) = ((\zeta_\sigma)^{\sigmagag}, \sigmagag).(j,\omega)
\end{eqnarray*}
for $j\in [k_i]$ and $\omega\in \Omega_i$. So $\rho(\sigma) = ((\zeta_\sigma)^{\sigmagag}, \sigmagag)\in S_n \wr_{\Omega} G$, as needed for (a).

If $[M_i:L_i]=n$, then $|\Psi_i|=n$, so $|\Phi_i| = n \cdot |\Omega_i|$. Then, since $\rho^*$ is injective and $\rho^*(\Phi_i) \subseteq [n]\times \Omega_i$, we get that $\rho^*(\Phi_i) = [n]\times \Omega_i$. Now since $H$ acts transitively on $\Phi_i$, its image $H'$ acts transitively on $[n]\times \Omega_i$. 

%
%
%
%
%

\end{proof}

\begin{corollary}\label{cor:comp-of-polyn}
Let $f,g\in K[X]$ be such that $f\circ g$ is separable, let $\Omega_f, \Omega_{f\circ g}$ be the sets of roots of $f$ and $f\circ g$, respectively, and let $n=\deg g$. Then there exists an embedding 
\[
\rho\colon \gal(f\circ g, K)\to S_n \wr_{\Omega_f} \gal(f,K)
\] 
that respects the action of the LHS on $\Omega_{f\circ g}$ and of the RHS on $[n]\times \Omega_f$. Moreover, if $\tilde{f}(X)$ is an irreducible factor of $f$ with set of roots $\tilde{\Omega}\subseteq \Omega$, then $\tilde{f}(g(X))$ is irreducible over $K$ if and only if the image of $\rho$ acts transitively on $[n]\times \tilde{\Omega}$. 
\end{corollary}

\begin{proof}
Let $f = \prod_{i=1}^r f_i$ be the factorization of $f$ into a product of irreducible polynomials. Let $\omega_i$ be a root of $f_i$, for $i=1, \ldots, r$. Let $L_i = K(\omega_i)$ and let $M_i = K(\nu_i)$, where $\nu_i$ is a root of $g(X) - \omega_i$. Since $\omega_i = g(\nu_i)$ we have $K\subseteq L_i \subseteq M_i$ and $[M_i:L_i]\leq n$. These extension  are separable because $f(g(X)) = \prod_{\omega\in \Omega_f} (g(X) - \omega)$ is separable. 
Now Proposition~\ref{lem:gal-tower} implies the assertion. 
\end{proof}

\begin{proof}[Proof of Proposition~\ref{prop:gene-comp-wreath}]
Let $G = \gal(f(X), K) \cong \gal(f(X) , K(\bfA))$. Let $g_i (\bfA,X) = g(\bfA,X) - \omega_i$, where $\Omega = \{\omega_1, \ldots, \omega_m\}$ are the roots of $f$. Then, since $f$ is separable, and since $f(g(\bfA,X)) = \prod_{i=1}^m g_i$, we get that $f\circ g$ is separable. Let $F$ be the splitting field of $f\circ g$ over $K(\bfA)$ and $L$ the splitting field of $f$ over $K$. By Corollary~\ref{cor:comp-of-polyn}, we can consider $H = \gal(F/K(\bfA))$ as a subgroup of $S_n \wr_{\Omega} G$. The kernel $\gal(F/L(\bfA))$ of the restriction map $H \to G$ coincides with $H\cap S_n^\Omega$.  

Proposition~\ref{prop:prod-sym-grps-oddchar} implies that $\gal(F \tilde{K}/\tilde{K} (\bfA)) = S_n^\Omega$, because it is the Galois group of the product of the $g_i$'s. (Here $\tilde{K}$ is an algebraic closure of $K$.) We have
\[
S_n^{\Omega} \geq H\cap S_n^{\Omega} = \gal(F/L(\bfA)) \geq \gal(F \tilde{K}/\tilde{K}(\bfA)) = S_n^\Omega,
\]
hence $\gal(F/L(\bfA)) = S_n^\Omega$, $F$ is regular over $L$, and $\gal(F/K(\bfA)) = S_n\wr_{\Omega} G$. 
\end{proof}

\section{Proof of Theorem \ref{thm:Ans_PAC}}
Let $K$ be a PAC field of characteristic $p\geq 0$, $f_1, \ldots, f_r$ irreducible non-associate polynomials, and $f = f_1\cdots f_r$. Let $g(\bfA,t) = t^n + A_1 t^{n-1} + \cdots + A_n$ be a polynomial with variable coefficients. Let $F$ (resp.\ $L$) be the splitting field of $f\circ g$ (resp.\ $f$) over $K(\bfA)$ (resp.\ $K$). Then, by Proposition~\ref{prop:gene-comp-wreath}, $F$ is regular over $L$ and $\gal(F/K(\bfA))= S_n \wr_{\Omega} \gal(L/K)$, where $\Omega = \coprod_{i=1}^r \Omega_i$ is the set of the roots of $f$, and $\Omega_i$ is the set of the roots of $f_i$, $i=1,\ldots, r$. Moreover, the restriction map $\alpha\colon \gal(F/K(\bfA)) \to \gal(L/K)$ coincides with the wreath product projection $S_n\wr_\Omega \gal(L/K)\to \gal(L/K)$. So 
\[
\mathcal{E}(f\circ g, \bbA^n_K) = (r_{L}\colon \gal(K) \to \gal(L/K), \alpha \colon S_n \wr_{\Omega} \gal(L/K)\to \gal(L/K)).
\]
(Here $r_{L} \colon \gal(K) \to \gal(L/K)$ is the restriction map.) 

By assumption, each $K(\omega_i)$ has a degree $n$ separable extension, for some $\omega_i\in \Omega_i$, $i=1, \ldots, r$. Thus Lemma~\ref{lem:gal-tower} (applied to $L_i=K(\omega_i)$ and $M_i$ the degree $n$ separable extension of $K(\omega_i)$) gives a homomorphism $\theta\colon \gal(K) \to S_n \wr_{\Omega} \gal(L/K)$ such that $\beta\circ \theta=r_{L}$.  So $\theta$ is a solution of $\mathcal{E}(f\circ g, \mathbb{A}^n_K)$. Moreover, Lemma~\ref{lem:gal-tower} gives that $\theta(\gal(K(\omega_i)) = \theta(\gal(K)) $ acts transitively on $[n]\times\Omega_i$, hence acts transitively on the roots of $f_i(g(\bfA,t))$. 

By Proposition~\ref{prop:geom-sol-PAC} there exists a Zariski dense set of $\mathfrak{p}\in \mathbb{A}^n(K)$ such that $\mathfrak{p}^* = \theta$. By Lemma~\ref{lem:factorizatiotype} $f_i(g(\bfA,t)) \mod \mathfrak{p}=f_i(g(\bfa,t))$ is irreducible over $K$.\qed

\section{Proof of Theorem~\ref{thm:large-finite-arithmetic-type-2}}
\subsection{Weak version}
Let $K$ be a pseudo finite field, i.e., a perfect PAC field with $\gal(K) = \hat{\mathbb{Z}}$. Fix $n,B$. Theorem~\ref{thm:Ans_PAC} asserts that $K$ satisfies  the elementary statement $\epsilon(K)$: 

\begin{quote}
If ($1+1 = 0 \Rightarrow n$ odd) and  (for every $f_1,\ldots, f_r \in K[X]$ irreducible, non-associate, and such that $\sum \deg f_i \leq B$), then there exists $(a_1, \ldots, a_n)\in K^n$ such that for $g(t) = t^n + a_1 t^{n-1} + \cdots + a_n$ and for every $i=1, \ldots, r$ we have $f_i(g(t))$ is irreducible. 
\end{quote} 

By Ax' theorem \cite[Proposition 20.10.4]{FriedJarden2008} this implies that for every $q=p^\nu$ sufficiently large $\mathbb{F}_q$ satisfies $\epsilon(\mathbb{F}_q)$. Hence we immediately get a weak form of Theorem~\ref{thm:large-finite-arithmetic-type-2}. 

In order to prove the full theorem basing on pseudo finite fields, we proof a more technical theorem than Theorem~\ref{thm:Ans_PAC} for pseudo finite fields, and then apply Ax's theorem. 

\subsection{Strong version}
We need an auxiliary lemma.
\begin{lemma}\label{lem:ncyclesinwrproduct}
Let $n$ be a positive integer, let $\Omega$ be a set of cardinality $\nu$, let $G = \langle \sigma \rangle$, for some  $\sigma \in \Sym(\Omega)$, let $\Omega = \coprod_{i=1}^r \Omega_i$ be the factorization of $\Omega$ to $G$-orbits, and let $H = S_n\wr_\Omega G$.  Then
\begin{enumerate}
\item\label{ncyclewr1}
 For each $\eta\in S_n^{\Omega}$, $\eta\sigma\in H$ acts of $[n]\times \Omega_i$.
 
\item\label{ncyclewr2} 
Let $T = \{\eta\sigma \in H \mid \eta\sigma \mbox{ acts transitively on } [n]\times \Omega_i, \forall i\}$. Then $S_n^\Omega$ acts transitively on $T$, and $|T| = \frac{(n!)^{\nu}}{n^r}$.
\end{enumerate}
\end{lemma}

\begin{proof}
Since $\eta \sigma.(k,\omega) = (\eta(\sigma.\omega).k,\sigma.\omega)$, we get \eqref{ncyclewr1}. 

Assume $\eta \sigma \in T$ and let $I = \langle \eta\sigma \rangle$. Then $I.(k,\omega) = [n]\times \Omega_i$, for each $(k,\omega)\in [n]\times \Omega_i$. Let $\zeta\in S_n^{\Omega}$. Then $\zeta^{-1} I \zeta . (k,\omega) = \zeta^{-1} I .(k',\omega) = \zeta^{-1} .[n]\times\Omega_i = [n]\times \Omega_i$. So $(\eta\sigma)^{\zeta}\in T$, and $S_n^{\Omega}$ acts on $T$.

It remains to prove that this action is transitive and to calculate $|T|$. We first prove it under the assumption that $G$ is transitive, i.e., that $r=1$. 

Let $\eta\in S_n^\Omega$ and $g = \eta\sigma\in S_n\wr_\Omega G$. We have 
\[
g^{k} = (\eta \sigma)^k = \eta \cdots \eta^{\sigma^{-(k-1)}} \sigma^k.  
\]
Since the order of $\sigma$ is $\nu$, it follows that $g^\nu$ is the stabilizer of $[n]\times \{\omega\}$, for some fixed $\omega\in \Omega$. Hence $g$ acts transitively on $[n]\times \Omega$ if and only if $g^\nu$ acts transitively on $[n]\times  \{\omega\}$. We have 
\[
g^\nu .(i,\omega) = \eta \cdots \eta^{\sigma^{-(\nu-1)}}.(i,\omega) = (\eta(\omega) \cdots \eta(\sigma^{(\nu-1)}(\omega))).i.
\]
So $g^\nu$ acts transitively on $[n]\times  \{\omega\}$ if and only if $\eta(\omega) \cdots \eta(\sigma^{(\nu-1)}(\omega))\in S_n$ is an $n$-cycle. We can 
choose $\eta(\omega), \ldots, \eta(\sigma^{\nu-2}(\omega))\in S_n$ to be arbitrary and $\eta(\sigma^{n-1})$ to be $(\eta(\omega) \cdots \eta(\sigma^{(\nu-1)}(\omega)))^{-1} \tau$, where $\tau$ is an $n$-cycle. Since there are $(n-1)!$ $n$-cycles in $S_n$, we have $(n!)^{\nu-1}(n-1)!=\frac{(n!)^{\nu}}{n}$ such choices. 

Let $g= \eta\sigma \in T$. Then by the latter paragraph, $\tau = \eta(\omega) \eta(\sigma.\omega) \cdots \eta(\sigma^{\nu^{-1}}. \omega)$ is an $n$-cycle.  Let $\zeta\in S_n^{\Omega}$; then we have
\begin{eqnarray*}
g = g^{\zeta} 
		&\Leftrightarrow
			&\eta \sigma = \zeta^{-1} \eta \zeta^{\sigma^{-1}} \sigma \Leftrightarrow\zeta = \eta \zeta^{\sigma^{-1}} \eta^{-1} \\
		&\Leftrightarrow
			&\zeta(\omega) = \zeta(\sigma.\omega)^{\eta^{-1}(\omega)} = \zeta(\sigma^{2}.\omega)^{\eta^{-1} (\sigma.\omega)\eta^{-1}(\omega)} = \cdots \\
			&&\qquad \; =\zeta(\sigma^{\nu}.\omega)^{\eta^{-1}(\sigma^{\nu^{-1}}.\omega)\cdots\eta^{-1} (\sigma.\omega)\eta^{-1}(\omega)} = \zeta(\omega)^{\tau^{-1}}.
\end{eqnarray*}
The latter condition is satisfied if and only if $\zeta(\omega)$ commutes with $\tau$ and $\zeta(\sigma^{j}.\omega)$ are determined by the equations. Since the centralizer of $\langle \tau \rangle$ in $S_n$ is $\langle\tau\rangle$ (\cite[Satz 6.5]{Huppert1967})m we get that there are $n$ such $\zeta$. Then the size of the orbit of $\eta\sigma$ is $\frac{(n!)^{\nu}}{n} = |T|$, so the action is transitive.

Next we consider the general case. Let $\sigma_i$ (resp.\ $G_i$) be the image of $\sigma$ (resp.\ $G$) under the map $\Sym(\Omega) \to \Sym(\Omega_i)$. The map $\eta\sigma \mapsto (\eta|_{\Omega_i}\sigma_i)_{i=1}^r$ defines an isomorphism
\[
\phi\colon S_n \wr_\Omega G \to \prod_{i=1}^r S_n\wr_{\Omega_i} G_i,
\]
that respects actions. Let $T^{(i)} = \{ \eta_i\sigma_i \in S_n\wr_{\Omega_i} G_i\mid \eta_i\sigma_i \mbox{ acts transitively on } [n]\times \Omega_i\}$. Then $\rho$ induces a bijection $\rho|_{T}\colon T\to \prod_{i=1}^r T^{(i)}$, and it follows that $|T| = \prod_{i=1}^r |T^{(i)}| = \frac{(n!)^{\nu}}{n^r}$. 

A $S_n^\Omega$-orbit $U$ of $T$ is mapped under $\phi$ to a product $\prod_{i} U_i$, where $U_i$ is an $S_n^{\Omega_i}$ orbit of $T^{(i)}$. Thus $U_i = T^{(i)}$, for every $i=1,\ldots ,r$, and $U=T$. 
\end{proof}

\begin{theorem}
\label{thm:pseudo-finite}
Let $n,B$ be fixed and let $K$ be a pseudo finite field of characteristic $p\geq 0$. If ($1+1 = 0 \Rightarrow n$ odd) and  (for every $f_1,\ldots, f_r \in K[X]$ irreducible, non-associate, and such that $\nu=\sum \deg f_i \leq B$), then there exists an absolutely irreducible smooth $K$-variety $\widehat{W}$ and a finite separable map $\pi \colon \widehat{W} \to \mathbb{A}^n$ of degree $n!^{\nu}$ such that for every $\mathbf{a} = (a_1, \ldots, a_n)\in K^n$ for which all $f_i(t^n + a_1 t^{n-1} + \cdots + a_n)$ are separable  we have  
\begin{enumerate}
\item 
\label{pf1}
$\pi^{-1} (\bfa) = 0$ or $n^r$,
\item 
\label{pf2}
$\bfa\in \pi(\widehat{W}(K))$ if and only if all $f_i(t^n + a_1 t^{n-1} + \cdots + a_n)$ are irreducible. 
\end{enumerate}
\end{theorem}

\begin{proof}
Let $g(\bfA,t) = t^n + A_1 t^{n-1} + \cdots +A_n$ be a generic polynomial, let $f=f_1\cdots f_r$. Then, by Proposition~\ref{prop:gene-comp-wreath}, the associated embedding problem $\mathcal{E}(f\circ g, \mathbb{A}^n)$ is
\[
\mathcal{E} =\mathcal{E}(f\circ g, \bbA^n_K) = (r_{L}\colon \gal(K) \to \gal(L/K), \alpha \colon S_n \wr_{\Omega} \gal(L/K)\to \gal(L/K)).
\]
(Here $r_{L} \colon \gal(K) \to \gal(L/K)$ is the restriction map.) 
Let $\Sigma$ be a (topological) generator of $\gal(K) \cong \hat{\mathbb{Z}}$; then $\sigma=\Sigma|_{L}$ is a generator of $\gal(L/K)$. 
A homomorphism $\theta\colon \gal(K)\to S_n\wr_{\Omega} \gal(L/K)$ is a solution of $\mathcal{E}$ if and only if $\theta(\Sigma) = \eta_\theta \sigma$, for some $\eta\in S_n^{\Omega}$. The image of $\theta$ acts transitively on $[n]\times \Omega_i$ if and only if $\eta_\theta\sigma$ acts transitively on $[n]\times \Omega_i$, $i=1,\ldots, r$. Thus by Lemma~\ref{lem:ncyclesinwrproduct}, there exists a unique  $\ker\alpha$-inner-automorphism class of solutions, say $\Theta$,  of cardinality $\frac{(n!)^{\nu}}{n^r}$.      

Let $V = \mathbb{A}^n$ and $W=V_{f\circ g}$. For $\mathfrak{p} = \bfa = (a_1,\ldots, a_n)\in K^n$ we have $f\circ g \mod \mathfrak{p}=  f(g(\bfa,t)$ and $\mathfrak{p}$ is \'etale in $W$ if and only if $f(g(\bfa,t))$ is separable, if and only if all $f_i(g(\bfa,t))$ are separable. The action of $\gal(K)$ on the roots of $f_i(g(\bfa,t))$ is the same as the action of $\mathfrak{P}^*$ on the roots of $f_i(g(\bfA,t))$, up to labeling of the roots, where $\mathfrak{P}^*\in \mathfrak{p}^*$ (by Eq.~\ref{eq:geo-sol-act} that defines geometric solutions). 

Therefore, $f_i(g(\bfa,t))$ is irreducible if and only if the image of $\mathfrak{P}^*\in \mathfrak{p}^*$ acts transitively on the roots of $f_i(g(\bfA,t))$. By Lemma~\ref{lem:ncyclesinwrproduct} the latter holds true if and only if the action of $[n]\times \Omega_i$ is transitive. Equivalently, all $f_i(g(\bfa,t))$ are irreducible if and only if $\mathfrak{p}^* = \Theta$.  

By Proposition~\ref{prop:geo-sol-rat-place} there exists an absolutely irreducible smooth $K$-variety $\widehat{W}$ and a finite separable map $\pi\colon \widehat{W}\to \mathbb{A}^n$ of degree $|\ker\alpha| = (n!)^{\nu}$ such that for $\mathfrak{p} = \bfa= (a_1,\ldots , a_n)$ that is \'etale in $W$ we have $\mathfrak{p}^* = \Theta$ if and only if $\mathfrak{p}\in \pi(\widehat{W}(K))$ and $\pi^{-1}(\bfa) = 0$ or $\frac{|\ker \alpha|}{|\Theta|} = n^r$. 
This proves that $f_i(g(\bfa,t))$ is irreducible for every $i$ if and only if $\mathfrak{p}^* = \theta$ for $\mathfrak{p}=\bfa$ if and only if $\mathfrak{p}\in \pi(\widehat{W}(K))$, and the proof is done. 
\end{proof}

We are now ready to prove Theorem~\ref{thm:large-finite-arithmetic-type-2}. 

\begin{proof}[Proof of Theorem~\ref{thm:large-finite-arithmetic-type-2}]
Since the assertion of Theorem~\ref{thm:pseudo-finite} is elementary, it holds for large finite fields. In particular, there is a  variety as stated.

$h(t) = f(t^n + a_1 t^{n-1} + \cdots + a_n)$ has a double root if and only if $\gcd(h,h')\neq 1$, which is a is a closed condition on the coefficients $a_1,\ldots, a_n$. Thus it contributes to the error term $O_B(q^{n-1})$. 

By the Lang-Weil estimates, we have that $|\widehat{W}(\mathbb{F}_q)| = q^{n} + O_{n,B}(q^{n-\frac12})$ (the error term depends on the degree of $\pi$ which equals $|\ker \alpha| = (n!)^\nu$, hence depends on $n$ and $\nu\leq B$). Thus, by Theorem~\ref{thm:pseudo-finite}, there are exactly
\[
\frac{|\widehat{W}(K)|}{n^r}  + O_{B}(q^{n-1}) = 
\frac{1}{n^r}q^n  + O_{n,B}(q^{n-\frac{1}{2}}) 
\]
$\bfa\in \mathbb{F}_q^n$ for which all $f_i(t^n + a_1 t^{n-1} + \cdots + a_n)$ are irreducible.
\end{proof}

\section{Concluding remarks}
The results of this paper can be generalizes in few ways.

\begin{enumerate}
\item In Theorem~\ref{thm:large-finite-arithmetic-type-2} one can consider specializing $X\mapsto g(t)$ in $f_i(X)$ such that the $f_i(g(t))$ will have a given factorization type, not only irreducible, as is done in \cite{Pollack2008}. Clearly not every factorization type of $f_i(g(t))$ can occur, since the irreducible factors are conjugated, so they have the same degree. Under this restriction one can get asymptotic for the number of such monic $g$'s of degree $n$. The only extra thing is to extend Lemma~\ref{lem:ncyclesinwrproduct} to other partitions. For brevity, the author decided to omit the exact formulation and  proof.

\item As mentioned in the introduction, Theorem~\ref{thm:Ans_PAC} can be extended to the family of fields that have a PAC extension satisfying \eqref{cond:exists_sep_ext}. The extra ingredient needed are double embedding problems and the lifting property appearing in \cite{Bary-Soroker2009PACEXT}. However since it is outside of the scope of this paper, this will be dealt in details somewhere else. 
\end{enumerate}

\bibliographystyle{amsplain}
\bibliography{/Users/LandH/Documents/lior/localtexmf/bib}
\end{document}